\documentclass{amsart}
\pdfoutput=1

\usepackage[mathscr]{eucal}
\usepackage{amssymb}
\usepackage{stmaryrd}
\usepackage[usenames,dvipsnames]{xcolor}
\usepackage{xspace}
\usepackage{amsthm}
\usepackage{bbold}
\usepackage{enumerate}

\usepackage{amsmath}
\usepackage{mathdots}
\usepackage{graphicx}
\usepackage{float}

\numberwithin{equation}{section}

\usepackage[all]{xy}
\SelectTips{cm}{}
\newdir{ >}{{}*!/-10pt/\dir{>}}

\usepackage[colorlinks=true,linkcolor={Black},citecolor={Black},urlcolor={Black}]{hyperref}
\definecolor{azure(colorwheel)}{rgb}{0.0, 0.5, 1.0}
\definecolor{amber}{rgb}{1.0, 0.49, 0.0}

\usepackage[nameinlink]{cleveref}

\crefname{Thm}{Theorem}{Theorems}
\crefname{Thm*}{Theorem}{Theorems}
\crefname{Rem}{Remark}{Remarks}
\crefname{Prop}{Proposition}{Propositions}
\crefname{Cor}{Corollary}{Corollaries}
\crefname{Cons}{Construction}{Constructions}
\crefname{Exa}{Example}{Examples}
\crefname{Lem}{Lemma}{Lemmas}
\crefname{Rec}{Recollection}{Recollections}
\crefname{subsection}{Subsection}{Subsections}

\swapnumbers 
\newtheorem{Cor}[equation]{Corollary}
\newtheorem{Lem}[equation]{Lemma}
\newtheorem{Prop}[equation]{Proposition}
\newtheorem{Thm}[equation]{Theorem}
\newtheorem*{Thm*}{Theorem}

\theoremstyle{remark}
\newtheorem{Def}[equation]{Definition}

\newtheorem{Exa}[equation]{Example}

\newtheorem{Rem}[equation]{Remark}
\newtheorem{Rec}[equation]{Recollection}

\newcommand{\nc}{\newcommand}
\nc{\dmo}{\DeclareMathOperator}

\dmo{\con}{con}
\dmo{\cone}{cone}
\dmo{\Der}{D}
\dmo{\geom}{gm}
\dmo{\DAM}{DAM}
\dmo{\DATM}{DATM}
\dmo{\DM}{DM}
\dmo{\DPerm}{DPerm}
\dmo{\ev}{ev}
\dmo{\Hm}{H}
\dmo{\id}{id}
\dmo{\Id}{Id}
\dmo{\Img}{Im}
\dmo{\Infl}{Infl}
\dmo{\Komp}{K}
\dmo{\Ker}{Ker}
\dmo{\Mod}{Mod}
\dmo{\modname}{mod}
\dmo{\per}{per}
\dmo{\Per}{Per}
\dmo{\perm}{perm}
\dmo{\Perm}{Perm}
\dmo{\Perf}{Perf}
\dmo{\proj}{proj}
\dmo{\Proj}{Proj}
\dmo{\Qcoh}{Qcoh}
\dmo{\Res}{Res}
\dmo{\SH}{SH}
\dmo{\smallb}{b}
\dmo{\smallperf}{perf}
\dmo{\Spc}{Spc}
\dmo{\Spec}{Spec}
\dmo{\Spech}{\Spec^{h}}
\dmo{\stmod}{stmod}
\dmo{\StMod}{StMod}
\dmo{\stperm}{stperm}
\dmo{\StPerm}{StPerm}
\dmo{\subname}{Sub}
\dmo{\supp}{supp}
\dmo{\Supp}{Supp}
\dmo{\End}{End}
\dmo{\kosname}{kos}
\dmo{\Ind}{Ind}
\dmo{\Sub}{Sub}
\dmo{\incl}{incl}
\dmo{\Inj}{Inj}
\dmo{\Desc}{Desc}

\nc{\adj}{\dashv}
\nc{\aka}{{a.\,k.\,a.}\ }
\nc{\calU}{\mathcal{U}}
\nc{\cat}[1]{\mathscr{#1}}
\nc{\cC}{\cat{C}}
\nc{\cJ}{\cat{J}}
\nc{\cK}{\cat{K}}
\nc{\cL}{\cat{L}}
\nc{\cM}{\cat{M}}
\nc{\colim}{\mathop{\mathrm{colim}}}
\nc{\compl}{\complement}
\nc{\cO}{\cat{O}}
\nc{\cP}{\cat{P}}
\nc{\cQ}{\cat{Q}}
\nc{\cS}{\cat{S}}
\nc{\cT}{\cat{T}}
\nc{\cV}{\cat{V}}
\nc{\Db}{\Der_{\smallb}}
\nc{\Dperf}{\Der_{\smallperf}}
\nc{\eg}{{\sl e.g.}\@\xspace}
\nc{\gm}{\mathfrak{m}}
\nc{\gn}{\mathfrak{n}}
\nc{\gp}{\mathfrak{p}}
\nc{\gq}{\mathfrak{q}}
\nc{\ideal}[1]{\langle #1\rangle}
\nc{\ie}{{\sl i.e.}\@\xspace}
\nc{\ihom}{{\mathsf{hom}}} 
\nc{\into}{\mathop{\rightarrowtail}}
\nc{\inv}{^{-1}}
\nc{\isoto}{\overset{\sim}{\,\to\,}}
\nc{\Kb}{\Komp_{\smallb}}
\dmo{\Locname}{Loc}
\nc{\Loc}[1]{\Locname(#1)}
\nc{\Loctens}[1]{\Locname^{\otimes}(#1)}
\nc{\loccit}{{\sl loc.\ cit.}\xspace}
\nc{\Mid}{\,\big|\,}
\nc{\mmod}[1]{\modname(#1)}
\nc{\normal}{\triangleleft}
\nc{\normaleq}{\trianglelefteqslant}
\nc{\potimes}[1]{^{\otimes #1}}
\nc{\qcO}{\mathcal{QO}}
\nc{\FF}{\mathbb{F}}
\nc{\QQ}{\mathbb{Q}}
\nc{\RR}{\mathbb{R}}
\nc{\sbull}{{\scriptscriptstyle\bullet}}
\nc{\SET}[2]{\big\{\,#1\Mid#2\,\big\}}
\nc{\sminus}{\smallsetminus}
\nc{\SpcK}{\Spc(\cK)}
\nc{\SpcL}{\Spc(\cL)}
\nc{\SpcT}{\Spc(\cT^c)}
\nc{\SpcS}{\Spc(\cS^c)}
\nc{\SpSp}{\mathsf{Spec}}
\nc{\To}{\Rightarrow}
\nc{\too}{\mathop{\longrightarrow}\limits}
\nc{\unit}{\mathbb{1}}
\nc{\ee}{\mathbb{e}}
\nc{\ff}{\mathbb{f}}
\nc{\Weyl}[2]{{#1}/\!\!/{#2}}
\nc{\WGH}{\Weyl{G}{H}}
\nc{\WGHi}{\Weyl{G}{H_i}}
\nc{\WGK}{\Weyl{G}{K}}
\nc{\wX}{\overline{X}{}^{\calU}}
\nc{\xto}[1]{\xrightarrow{#1}}
\nc{\bbA}{\mathbb{A}}
\nc{\bbF}{\mathbb{F}}
\nc{\bbR}{\mathbb{R}}
\nc{\bbZ}{\mathbb{Z}}
\nc{\onto}{\mathop{\twoheadrightarrow}}
\nc{\eqperf}{equivariantly perfect}
\nc{\Keqperf}[1]{\cK_{\textrm{eq-perf}}(#1)}
\nc{\TU}{\cT|_{U}}
\nc{\SV}{\cS|_{V}}
\nc{\fUV}{f^*|_{U}^{V}}
\nc{\Kac}{\cK_{\textrm{ac}}}

\date{2026 April 16}

\author{Paul Balmer}
\address{Paul Balmer, UCLA Mathematics Department, Los Angeles, CA 90095, USA}
\email{balmer@math.ucla.edu}
\urladdr{https://www.math.ucla.edu/~balmer}

\author{Martin Gallauer}
\address{Martin Gallauer, Mathematics Institute, University of Warwick}
\email{martin.gallauer@warwick.ac.uk}
\urladdr{https://mgallauer.warwick.ac.uk}

\hypersetup{pdfauthor={Paul Balmer and Martin Gallauer},pdftitle={Permutation, stabilization and decomposition}}

\usepackage{xcolor}
\nc{\paul}[1]{{\color{Orchid}{#1}}}
\nc{\martin}[1]{{\color{brown}{#1}}}

\begin{document}


\title{Permutation, stabilization and decomposition}

\begin{abstract}
Informed by our understanding of the tt-geometry of permutation modules, we investigate the proper definition of the `stable permutation category' of a finite group.
Then we prove that this category decomposes over cyclic and generalized quaternion groups and only in those cases.
\end{abstract}


\maketitle

\ \hfill \textit{Dedicated to Dave Benson on the occasion of his seventieth birthday}
\medbreak

\section{Introduction}
\label{sec:intro}%

For the whole paper, $G$ is a finite group and $k$ a field of characteristic~$p>0$.
We introduce the \emph{stable permutation category} of~$G$ with coefficients in~$k$
\[
\StPerm(G;k).
\]
It is to the usual stable module category $\StMod(kG)$, of $kG$-modules modulo projectives, what permutation modules are to general modules.
However $\StPerm(G;k)$ is \emph{not} defined as the additive quotient of the category of permutation modules by the subcategory of projectives.
We explain in~\Cref{Rem:thick(perm)=stab} why this definition would not be very interesting.
In fact, our stable permutation category~$\StPerm(G;k)$ is not even a subcategory of $\StMod(kG)$ but rather an intermediate localization between the \emph{derived category of permutation modules}~$\DPerm(G;k)$ introduced in~\cite{BalmerGallauer23b} and the stable module category~$\StMod(kG)$. Let us remind the reader.

We write $\perm(G;k)^\natural$ for the additive subcategory of $p$-permutation $kG$-modules inside the abelian category~$\mmod{kG}$ of finitely generated~$kG$-modules (\Cref{Rec:perm}). The bounded homotopy category~$\Kb(\perm(G;k)^\natural)$ is the compact part of~$\DPerm(G;k)$. On the other hand, the bounded derived category $\Db(kG):=\Db(\mmod{kG})$ is the compact part of $\Komp(\Inj(kG))$ by Krause~\cite{Krause05b}.
The stable permutation category~$\StMod(G;k)$ fits in a commutative square of localizations:
\begin{equation}
\label{eq:four-tt-cats}%
\vcenter{\xymatrix@C=2em{\DPerm(G;k) \ar@{->>}[r]^-{\Upsilon} \ar@{->>}[d]_-{p}
& \Komp(\Inj(kG)) \ar@{->>}[d]^-{q}
\\
{\StPerm(G;k)} \ar@{->>}[r]^-{\bar\Upsilon}
& \StMod(kG)
}}
\qquad
\vcenter{\xymatrix@C=2em{\Kb(\perm(G;k)^\natural) \ar@{->>}[r]^-{\Upsilon} \ar@{->>}[d]_-{p}^(.65){\natural}
& \Db(kG) \ar@{->>}[d]^-{q}
\\
{\stperm(G;k)} \ar@{->>}[r]^-{\bar\Upsilon}
& \stmod(kG).\!\!
}}\kern-1em
\end{equation}
The left-hand square displays big categories and the right-hand one the corresponding compact parts.
In each square the left-hand column involves permutation modules, the right-hand column involves arbitrary $kG$-modules, and the horizontal functors~$\Upsilon$ and~$\bar\Upsilon$ are induced by the canonical inclusion $\perm(G;k)^\natural\hookrightarrow\mmod{kG}$.
The vertical functor~$q$ is the finite localization away from perfect complexes, by Rickard~\cite{Rickard89}. We \emph{define} $\StPerm(G;k)$ as the finite localization of $\DPerm(G;k)$ away from so-called \emph{\eqperf} complexes; the latter are those bounded complexes of $p$-permutation modules that are not only perfect in the derived category $\Db(kG)$ but such that \emph{all their modular $H$-fixed points} (\aka Brauer quotients) are perfect, for all $p$-subgroups~$H\le G$. See details in \Cref{sec:stperm}.

We prove that this stable permutation category is also the localization of the derived permutation category $\DPerm(G;k)$ on the open complement of the finitely many closed points in its spectrum.
Conjecturally, $\StPerm(G;k)$ is also the localization of $\DPerm(G;k)$ on its `periodic locus' in the sense of~\cite{gallauer:periods}. This fact is known to hold for a class of groups containing $p$-groups but remains a conjecture in full generality. See details in \Cref{sec:tt-justifications}.

\medbreak

With the definition established and justified by these alternate interpretations, we investigate when this stable permutation category $\StPerm(G;k)$ is indecomposable as a tt-category. This question might come as a surprise since the other three types of categories in~\eqref{eq:four-tt-cats} are always indecomposable, for instance because the ring of endomorphisms of their $\otimes$-unit $\End(\unit)=k$ is indecomposable.
And yet the stable permutation category can indeed decompose, in very special cases.
\begin{Thm}[\Cref{Thm:indecomposability}]
\label{Thm:indecomposability-intro}%
The stable permutation category $\StPerm(G;k)$ is indecomposable unless the $p$-Sylow of $G$ is cyclic or generalized quaternion.
\end{Thm}

We obtain this result by proving that the spectrum is connected as a topological space. Conversely, for $G$ with cyclic or generalized quaternion $p$-Sylow, the spectrum of $\stperm(G;k)$ is disconnected and this forces $\StPerm(G;k)$ to be a product of two tt-categories, or more.
Amusingly, the connected components of the spectrum that appear in those cases are reminiscent of spectra of other tt-categories and this observation led us to \emph{predict} what the decompositions of~$\StPerm(G;k)$ should be at the categorical level. In both cases, our geometric guess was confirmed:
\begin{Thm}[\Cref{Thm:cyclic}]
\label{Thm:cyclic-intro}%
The stable permutation category $\StPerm(C_{p^n};k)$ of the cyclic group of order~$p^n$ is tt-equivalent to the product of the $n$ usual stable module categories $\StMod(kC_p)\times\cdots \times\StMod(kC_{p^i})\times \cdots \times\StMod(kC_{p^n})$.
\end{Thm}
\begin{Thm}[\Cref{Thm:quaternion}]
\label{Thm:quaternion-intro}%
For $p=2$ and $n\ge 3$, the stable permutation category $\StPerm(Q_{2^n};k)$ of the
generalized quaternion group~$Q_{2^n}$ of order~$2^n$ is tt-equivalent to the product
$\StPerm(D_{2^{n-1}};k)\times \StMod(kQ_{2^n})$ of the stable permutation category of the dihedral group~$D_{2^{n-1}}$ of order~$2^{n-1}$ and the usual stable module category of~$Q_{2^n}$. Both of these factors are indecomposable.
\end{Thm}

The organization of the paper is the following. \Cref{sec:stperm} gives the definition of~$\StPerm(G;k)$, with equivalent formulations and motivation from the perspective of tt-geometry.
We gather some basic properties in \Cref{sec:general}, including Mackey 2-functoriality. We also prove that the modular fixed-points (Brauer quotients) survive on stable permutation categories and we prove a `Colimit Theorem' reducing the description of the spectrum of $\StPerm(G;k)$ to elementary abelian \emph{subquotients} of~$G$, analogous to the result for~$\DPerm(G;k)$ in~\cite{tt-perm}.
In \Cref{sec:indecomposability}, we prove \Cref{Thm:indecomposability-intro}, after some group-theoretic preparation in~\Cref{sec:bottleneck}.
In \Cref{sec:cyclic-quaternion} we give the decompositions of \Cref{Thm:cyclic-intro,Thm:quaternion-intro}.
In fact, similar decompositions hold for any finite group~$G$ whose $p$-Sylow is cyclic or generalized quaternion.


\section{Definition via \eqperf\ complexes}
\label{sec:stperm}%

We write $\mmod{kG}$ for the abelian category of finitely generated $kG$-modules and write $\Db(kG)$ for its bounded derived category.
Before turning to permutations, let us remind the reader of the usual stable \emph{module} category.
\begin{Rec}
\label{Rec:stmod}%
The \emph{stable module category} $\stmod(kG)$ is the additive quotient of the category~$\mmod{kG}$ by the subcategory of projective modules.
(Additive quotients by a subcategory are obtained by keeping the same objects and by modding out the morphisms that factor via an object of the subcategory.)
It is triangulated by Happel~\cite{Happel88}. In fact, by Rickard~\cite{Rickard89}, we have a triangular-equivalence
\begin{equation}
\label{eq:Rickard}%
\stmod(kG)\cong \Db(kG)/\Dperf(kG)
\end{equation}
with the Verdier quotient of the derived category $\Db(kG)$ by the thick subcategory of perfect complexes.
The category $\Db(kG)/\Dperf(kG)$ is effortlessly triangulated, as any Verdier quotient. In contrast, we need $kG$ to be Frobenius to obtain the equivalence~\eqref{eq:Rickard} and to show that the additive quotient $\stmod(kG)$ is triangulated.
\end{Rec}

The problem we want to address is to find the correct analogue of the stable module category $\stmod(kG)$ in the context of \emph{permutation modules}.
We assume that the reader is aware of the importance of permutation modules in representation theory and beyond, \eg in the theory of motives. See~\cite{BalmerGallauer23b} if necessary.

\begin{Rec}
\label{Rec:perm}%
The full additive subcategory $\perm(G;k)$ of~$\mmod{kG}$ of \emph{permutation} $kG$-modules consists of those isomorphic to~$k(X)$ for a finite $G$-set~$X$.
We slightly enlarge it by including direct summands~$\perm(G;k)^\natural\subseteq\mmod{kG}$; the latter are called \emph{$p$-permutation} or \emph{trivial source} $kG$-modules.
(If $G$ is a $p$-group then $\perm(G;k)^\natural=\perm(G;k)$.)
As $\perm(G;k)^\natural$ is just an additive category, its bounded derived category is simply its bounded homotopy category, no need to invert quasi-isomorphisms. We shall denote it~$\cK(G)$ as we did in~\cite{tt-perm}:
\[
\cK(G):=\Db(\perm(G;k)^\natural)=\Kb(\perm(G;k)^\natural).
\]
The `big' \emph{derived category of permutation modules} $\DPerm(G;k)$ of~\cite{tt-perm} is a rigidly-compactly generated tt-category whose compact part is the above~$\cK(G)$.
This `Ind-completion' can be realized for instance inside the unbounded homotopy category~$\Komp(\Perm(G;k))$ of not necessarily finitely generated permutation modules, as the localizing subcategory generated by~$\perm(G;k)$.
Its rigid-compact objects is indeed the subcategory $\DPerm(G;k)^c=\textrm{thick}(\perm(G;k))=\cK(G)$.

Alternatively, $\DPerm(G;k)$ is the homotopy category of modules over the Bredon cohomology spectrum~$H\underline{k}$ in genuine $G$-spectra, see~\cite{fuhrmann2025modularfixedpointsequivariant}.
\end{Rec}

\begin{Rem}
\label{Rem:perm-mod-loc}%
The inclusion $\perm(G;k)^\natural\subseteq\mmod{kG}$ does not yield an inclusion on derived categories.
Quite the opposite, the canonical functor, that we denote $\Upsilon\colon\cK(G)\to \Db(kG)$, is actually a Verdier quotient by~\cite[Theorem~5.13]{BalmerGallauer23a}.
As $\stmod(kG)$ is a further Verdier quotient of~$\Db(kG)$, we immediately obtain:
\end{Rem}

\begin{Cor}
\label{Cor:thick(perm)=stab}%
The thick subcategory of $\stmod(kG)$ generated by permutation modules is the whole~$\stmod(kG)$.
\qed
\end{Cor}
\begin{Rem}
\label{Rem:thick(perm)=stab}%
Since free modules are permutation, the category $\perm(G;k)^\natural$ contains the subcategory $\proj(kG)$ of finitely generated projective $kG$-modules.
One can then form the \emph{additive} quotient
\begin{equation}
\label{eq:stab-caca}%
\frac{\perm(G;k)^\natural}{\proj(kG)}.
\end{equation}
This additive category has been called the `stable category of $p$-permutation modules' in (unpublished) literature. We shall avoid this terminology.
The additive quotient~\eqref{eq:stab-caca} is clearly a full subcategory of the stable module category~$\stmod(kG)$.
However \eqref{eq:stab-caca} is not triangulated in any evident way.
In fact \eqref{eq:stab-caca} generates the whole $\stmod(kG)$ as a thick triangulated subcategory by \Cref{Cor:thick(perm)=stab}.
\end{Rem}

In view of the above discussion, we propose to define $\stperm(G;k)$ as a suitable localization of the bounded derived category of $p$-permutation modules~$\cK(G)$. We present justifications in the subsequent \Cref{sec:tt-justifications}.

\begin{Rec}
\label{Rec:modular-fixed-pts}%
Let $H\le G$ be a $p$-subgroup and write $\WGH=N_G(H)/H$ for the Weyl group of~$H$ in~$G$.
There exists (see~\cite[\S\,5]{tt-perm}) a tensor functor
\[
\Psi^H\colon \perm(G;k)^\natural\to \perm(\WGH;k)^\natural
\]
characterized by the property that $\Psi^H(kX)\cong k(X^H)$ for every finite $G$-set~$X$.
It is called \emph{modular $H$-fixed-points} or \emph{Brauer quotient}. It induces a tt-functor on homotopy categories (applying $\Psi^H$ degreewise on complexes) that we still denote~$\Psi^H\colon\cK(G)\to \cK(\WGH)$. Composed with the canonical localization of \Cref{Rem:perm-mod-loc} we obtain a tt-functor to the derived category of the Weyl group
\[
\check\Psi^H\colon \cK(G)\xto{\Psi^H}\cK(\WGH)\overset{\ \Upsilon\ }{\onto}\Db(k(\WGH)).
\]
\end{Rec}

\begin{Def}
\label{Def:G-perfect}%
A bounded complex~$C\in\cK(G)$ of $p$-permutation $kG$-modules is called \emph{\eqperf} if its image under modular $H$-fixed-points functor $\Psi^H(C)$ is a perfect complex over the Weyl group of~$H$, for every $p$-subgroup~$H\le G$.
This means that $\Psi^H(C)$ is \emph{quasi-isomorphic} to a bounded complex of finitely generated projective $k(\WGH)$-modules, or in formula~$\check\Psi^H(C)\in\Dperf(k(\WGH))$, for every~$H$.
Let us denote by
\[
\Keqperf{G}=\DPerm(G;k)^c_{\textrm{eq-perf}}
\]
the subcategory of~$\cK(G)=\DPerm(G;k)^c$ consisting of \eqperf\ complexes.
Equivalently,
\[
\Keqperf{G}=\Ker\left(\cK(G)\xto{(\Psi^H)}\prod_{H\le G}\Db(k(\WGH))\xto{\eqref{eq:Rickard}}\prod_{H\le G}\stmod(k(\WGH))\right)
\]
can be written as the kernel of a tt-functor (where $H$ runs through the (conjugacy classes of) $p$-subgroups).
In particular, it is a tt-ideal.
\end{Def}

\begin{Exa}
\label{Exa:naive-G-perfect}%
Every bounded complex of projective $kG$-modules is \eqperf. Indeed, $\Psi^H(kG)$ is $kG$ for $H=1$ and zero for $H\neq 1$.
\end{Exa}
\begin{Exa}
\label{Exa:Cp-G-perfect}%
Let $G=C_p$ and $C=0\to k\to kC_p\to kC_p\to k\to 0$ the usual acyclic complex. For $H=1$, the complex $\Psi^1(C)=C$ is quasi-isomorphic to zero hence it is perfect as complex of~$kG$-modules. And for $H=C_p$, the complex $\Psi^{C_p}(C)=0\to k\to 0\to 0\to k\to 0$ is perfect over the trivial group~$\WGH=1$, as every complex of $k$-vector spaces is. It follows that $C$ is \eqperf.
\end{Exa}

\begin{Rem}
\label{Rem:G-perf}%
Let us dispel any possible confusion between the \eqperf\ complexes of \Cref{Def:G-perfect} and two adjacent notions
\[
\Kb(\proj(kG))\ \subseteq\ \Keqperf{G}\ \subseteq\ \Upsilon\inv(\Dperf(kG)).
\]
\begin{enumerate}[\rm(a)]
\item
Every bounded complex of projective $kG$-modules is \eqperf\ (\Cref{Exa:naive-G-perfect}). It follows that, inside~$\cK(G)$, we have an inclusion of tt-ideals $\ideal{kG}=\Kb(\proj(kG)) \subseteq\Keqperf{G}$. It is usually a proper inclusion, as shown in~\Cref{Exa:Cp-G-perfect}.
\smallbreak
\item
\label{it:G-perf-2}%
By definition (take $H=1$ in \Cref{Def:G-perfect}), every \eqperf\ complex is perfect as a complex of $kG$-modules. The converse is true for groups with very small $p$-Sylow but is false in general. Indeed, suppose that $G$ has order divisible by~$p^2$. Pick $H$ a copy of~$C_p$ in the center of a $p$-Sylow of~$G$. Our assumption implies that $p$ divides the order of~$\WGH$.
Let $C={}^\otimes\Ind_1^G(0\to k\xto{1}k\to 0)$ be the `Koszul object' of~\cite[3.15]{tt-perm}, where it is denoted~$\kosname_G(1)$.
This complex is acyclic, hence trivially perfect as a complex of~$kG$-modules. However, $\Psi^H(C)$ generates~$\cK(\WGH)$ as a tt-ideal, by \cite[Lemma~5.21 (for $K=1$)]{tt-perm}. Since  $\WGH$ has order divisible by~$p$, the category~$\cK(\WGH)$ cannot consist only of perfect complexes (for $\stmod(k(\WGH))\neq0$). Therefore $\Psi^H(C)$ is not perfect over~$\WGH$.
\end{enumerate}
\end{Rem}

We are now ready for our central definition.
\begin{Def}
\label{Def:stperm}%
The \emph{stable permutation category} of~$G$ is the idempotent-completion of the Verdier quotient of the derived permutation category~$\cK(G)=\Kb(\perm(G;k)^\natural)$ by the tt-ideal of \eqperf\ complexes (\Cref{Def:G-perfect})
\[
\stperm(G;k)=\bigg(\frac{\cK(G)}{\Keqperf{G}}\bigg)^\natural.
\]
It is therefore an idempotent-complete rigid tensor-triangulated category in such a way that the $\natural$-localization functor $p\colon\cK(G)\to \stperm(G;k)$ is a tt-functor.

Let $\StPerm(G;k)=\DPerm(G;k)/\Loc{\Keqperf{G}}$ be the corresponding finite localization of the `big' derived permutation category $\DPerm(G;k)$ at the tt-ideal $\Keqperf{G}$ of its compacts~$\DPerm(G;k)^c=\cK(G)$. We baptize $\StPerm(G;k)$ the \emph{`big' stable permutation category} of the group~$G$, with coefficients in~$k$.
By Neeman-Thomason localization~\cite{Neeman92b}, we can identify~$\StPerm(G;k)^c$ with~$\stperm(G;k)$ -- hence the idempotent-completion in the latter.
\end{Def}

\begin{Rem}
We do not know whether the idempotent-completion is necessary in the definition of~$\stperm(G;k)$.
The definition of $\stmod(kG)$, either as an additive (\Cref{Rec:stmod}) or triangulated quotient through Rickard's Theorem~\eqref{eq:Rickard}, does not involve an idempotent-completion because one can show that it is already idempotent-complete.
It is conceivable that the same holds true for permutation modules but it would require an argument.
\end{Rem}

\begin{Prop}
\label{Prop:stperm->>stmod}%
We have the commutative diagrams of tt-functors~\eqref{eq:four-tt-cats}.
The stable module category $\StMod(G;k)$ is the finite localization of the stable permutation category $\StPerm(G;k)$ away from the tt-ideal of compacts consisting of complexes of $p$-permutation modules that are perfect as complexes of $kG$-modules.
\end{Prop}
\begin{proof}
By \Cref{Rem:G-perf}\,\eqref{it:G-perf-2}, the \eqperf\ complexes are perfect. This means that the kernel of~$p$ is sent to zero by~$q\circ \Upsilon$ and therefore $q\circ \Upsilon$ localizes along~$p$, to yield the unique factorization~$\bar\Upsilon$. The kernel of the localization~$\bar\Upsilon$ is then the image under~$p$ of the kernel of~$q\circ\Upsilon$, that is, the localizing subcategory generated by complexes of $p$-permutation modules that are perfect in~$\Db(kG)$.
\end{proof}

\section{Justifications of the definition}
\label{sec:tt-justifications}%

We want to justify our definition of the stable permutation category given in \Cref{sec:stperm}, from various tt-geometric perspectives.
We remind the reader of the stable module category precedent.

\begin{Rec}
By~\cite{BensonCarlsonRickard97} the spectrum of $\Db(kG)$ is homeomorphic, via the comparison map, to the homogeneous spectrum of the cohomology (see also~\cite[Proposition~8.5]{balmer:sss}):
\begin{equation}
\label{eq:Spc-Db}%
\Spc(\Db(kG))\isoto \Spech(\Hm^\sbull(G,k)).
\end{equation}
This space is local: Its unique closed point is the tt-ideal~$\cM=(0)$ of~$\Db(kG)$ on the left-hand side of~\eqref{eq:Spc-Db}, which corresponds to~$\Hm^+(G,k)$ on the right-hand side. In the language of projective algebraic geometry, $\Hm^+(G,k)$ is the `irrelevant' ideal and removing it yields the usual projective variety~$\Proj(\Hm^\sbull(G,k))$.
At the level of the derived category~$\Db(kG)$, the closed point~$\cM=(0)$ is exactly the support of the tt-ideal $\Dperf(kG)$ of perfect complexes. Removing this support, the open complement~$\Spc(\Db(kG))\sminus\{\cM\}$ becomes the spectrum of the corresponding localization $\Db(kG)/\Dperf(kG)$.
And the latter is the stable module category by~\eqref{eq:Rickard}.
Away from the `irrelevant' closed points, the homeomorphism~\eqref{eq:Spc-Db} restricts to
\begin{equation}
\label{eq:Spc-stmod}%
\Spc(\stmod(kG))\isoto \Proj(\Hm^\sbull(G,k)).
\end{equation}
In summary, the stable module category~$\stmod(kG)$ is obtained by a localization of the usual derived category $\Db(kG)$ away from the irrelevant ideal, \ie away from the unique closed point of its spectrum.
\end{Rec}

A similar pattern holds for the stable permutation category, with `local' replaced by `semi-local'. Let us recall some tt-geometry from~\cite{tt-perm}.
\begin{Rec}
\label{Rec:tt-perm}%
The main result of~\cite[\S\,7]{tt-perm} identified the spectrum of the tt-category $\cK(G)=\DPerm(G;k)^c$ of \Cref{Rec:perm}.
This spectrum admits a stratification by locally closed subsets indexed by conjugacy classes of $p$-subgroups
\begin{equation}
\label{eq:spcK-strata}%
\Spc(\cK(G))=\coprod_{(H)\in\Sub_p(G)/G}\cV_{\WGH}
\end{equation}
where $\cV_{\WGH}\cong\Spech(\Hm^*(\WGH;k))\cong\Spc(\Db(k(\WGH)))$ is the (extended) cohomological support variety associated with the Weyl group~$\WGH$, also known as the cohomological open in~$\Spc(\cK(\WGH))$. More precisely, for each $p$-subgroup~$H\le G$, the tt-functor~$\check\Psi^H\colon \cK(G)\to \Db(k(\WGH))$ of \Cref{Rec:modular-fixed-pts} induces a continuous map $\check\psi^H=\Spc(\check\Psi^H)$ that we prove to be injective~$\check\psi^H\colon\Spc(\Db(k(\WGH)))\hookrightarrow \Spc(\cK(G))$ and whose image is our~$\cV_{\WGH}$ in~\eqref{eq:spcK-strata}.
In particular, the closed point~$(0)$ of each~$\Spc(\Db(k(\WGH)))$ as in \Cref{Rec:stmod} gives a closed point
\[
\cM(H)=\cM_G(H):=\check\psi^H(0)=(\check\Psi^H)\inv(0)=\Ker(\check\Psi^H\colon \cK(G)\to \Db(k(\WGH)))
\]
that belongs to the stratum~$\cV_{\WGH}$. We proved in~\cite[Corollary~7.31]{tt-perm} that these $\cM(H)$ are all the closed points of~$\Spc(\cK(G))$.
In other words, $\cK(G)$ is `semi-local' in the sense that its spectrum admits only finitely many closed points~$\cM(H)$, one for every $G$-conjugacy class of $p$-subgroups~$H\in\Sub_p(G)$.
\end{Rec}
\begin{Prop}
\label{Prop:supp(G-perf)}%
An object~$C\in \cK(G)=\DPerm(G;k)^c$ is \eqperf\ in the sense of \Cref{Def:G-perfect} if and only if its support is contained in the subset~$\SET{\cM(H)}{H\in\Sub_p(G)}$ of closed points of $\Spc(\cK(G))$.
\end{Prop}
\begin{proof}
As with any tt-functor the support of the image $\check\Psi^H(C)$ of an object~$C$ is the preimage $(\check\psi^H)\inv(\supp(C))$ of its support under the induced map~$\check\psi^H=\Spc(\check\Psi^H)$.
It follows from~\eqref{eq:spcK-strata} that~$C$ has support in~$\SET{\cM(H)}{H\in\Sub_p(G)}$ if and only if every $\check\Psi^H(C)\in\Db(k(\WGH))$ has support in the closed point~$(0)$ of~$\Spc(\Db(\WGH))$, which is equivalent to $\check\Psi^H(C)$ being perfect over~$\WGH$ by \Cref{Rec:stmod}.
\end{proof}

\begin{Rem}
\label{Rem:K|U}%
\Cref{Def:stperm} is an instance of a very general tt-construction. For any Thomason subset~$Y\subseteq \SpcK$ of the spectrum of a tt-category~$\cK$, we form the \emph{localization of~$\cK$ on the complement~$U=Y^\compl$} by localization-idempotent-completion
\[
\cK|_U:=(\cK/\cK_Y)^\natural
\]
We typically do this for $Y$ closed with quasi-compact open complement~$U$. This tt-category~$\cK|_U$ comes with a $\natural$-localization tt-functor $\cK\to \cK|_U$ that we sometimes refer to as restriction on~$U$.
It induces an embedding $\Spc(\cK|_U)\hookrightarrow\SpcK$ whose image is~$U$. In short, $\Spc(\cK|_{U})=U$, as it should be.

When $\cK=\cT^c$ is the compact part of a `big' tt-category~$\cT$, it follows by~\cite{Neeman92b} that $\cT|_U:=\cT/\Loc{\cK_Y}$ remains a `big' tt-category whose compact part is the idempotent-completion of the corresponding Verdier quotient of compacts~$\cK/\cK_Y$, in other words $(\cT|_U)^c=(\cT^c)|_U$.

We apply this to~$\cT=\DPerm(G;k)$ and~$\cK=\cT^c=\cK(G)$ as in \Cref{Rec:perm} and to~$U\subset\Spc(\cK(G))$ the open complement of the closed points. It is quasi-compact for~$\Spc(\cK(G))$ is noetherian by~\cite[Proposition~9.1]{tt-perm}. The restriction of~$\cT$ on~$U$ is~$\cT|_U=\StPerm(G;k)$, with compact part~$\cK(G)|_U=\stperm(G;k)$ by \Cref{Prop:supp(G-perf)}.
It follows that the spectrum of the stable permutation category is the open~$U$ obtained from the semi-local $\Spc(\cK(G))$ by `puncturing out' the finitely many closed points:
\end{Rem}
\begin{Cor}
\label{Cor:Spc(stperm)}%
The spectrum of $\stperm(G;k)$ is homeomorphic to the open subspace of~$\Spc(\cK(G))$ complement of the closed points~$\SET{\cM(H)}{H\in\Sub_p(G)}$.
\qed
\end{Cor}

\begin{Exa}
\label{Exa:decomp-cyclic}%
Let~$G=C_{p^n}$ for $n\ge1$. Write $1=H_n<\cdots<H_1<H_0=G$ for the subgroups of~$G$, with $H_i$ of index~$p^i$.
Then by~\cite[Proposition~8.3]{tt-perm} the spectrum of~$\cK(C_{p^n})$ is the following space with $2n+1$ points:
\begin{equation}
\label{eq:SpcK-cyclic}%
W^n=\qquad
\vcenter{\xymatrix@R=1em@C=.7em{
{\scriptstyle\gm_0\kern-1em}
& {\bullet} \ar@{-}@[Gray] '[rd] '[rr] '[drrr]
&&
{\bullet}
&{\kern-1em\scriptstyle\gm_1}
&&{\scriptstyle\gm_{n-1}\kern-1em}&\bullet&&\bullet&{\kern-1em\scriptstyle\gm_n}
\\
&{\scriptstyle\gp_1\kern-1em}& {\bullet}&&& {\cdots}& \bullet\ar@{}[r]|(.7){\scriptstyle\gp_{n-1}^{\vphantom{I}}}\ar@{-}@[Gray] '[ru] '[rr] '[rrru] &&\bullet&{\kern-1em}\scriptstyle\gp_{n}}}
\end{equation}
The spectrum of the derived category of any non-trivial cyclic $p$-group (\eg, the quotients~$G/H_i$ for~$i>0$) is a Sierpi\'nski space $\Spc(\Db(C_{p^i}))=\{\gp,\gm\}$ with $\gp$ open and~$\gm$ closed. The stratum~$\cV_{G/H_i}$ of~\eqref{eq:spcK-strata} is the pairs~$\{\gp_i, \gm_i\}$ of~\eqref{eq:SpcK-cyclic} for $i=1\ldots,n$ and the left-most stratum boils down to~$\cV_{G/G}=\{\gm_0\}$ for~$i=0$.
Here $\gm_i=\Ker(\check\Psi^{H_i})=\cM(H_i)$ is the closed point corresponding to the $p$-subgroup~$H_i$. Removing all those closed points leaves a finite and discrete subspace
\[
\Spc(\stperm(C_{p^n};k))=\{\gp_1\}\sqcup\ldots\sqcup\{\gp_n\}.
\]
which is the spectrum of the stable permutation category by \Cref{Cor:Spc(stperm)}.
By general tt-geometry~\cite{Balmer07}, the fact that the spectrum of~$\cK=\stperm(C_{p^n};k)$ is disconnected $\SpcK=U_1\sqcup \cdots \sqcup U_n$ forces the rigid idempotent-complete tt-category~$\cK$ to be the product $\cK\cong\cK|_{U_1}\times\cdots \times\cK|_{U_n}$ of its local categories on each~$U_i$. We shall compute those components in \Cref{Thm:cyclic}.
\end{Exa}

\begin{Exa}
\label{Exa:DPerm-V4}%
Let $G=V_4=C_2^{\times 2}$ be the Klein-four group and $p=2$.
The stratification~\eqref{eq:spcK-strata} consists of 5 strata: a singleton (for $H=V_4$), two Sierpinski spaces (for each $H<V_4$ of index~$2$), and an ``extended'' projective line~$\bar{\mathbb{P}}^1_{\!\!k}$, that is, a projective line with a unique closed point added on top (for $H=1$).
Removing all the closed points, \Cref{Cor:Spc(stperm)} tells us that $\Spc(\stperm(G;k))$ is a~$\mathbb{P}^1_{\!\!k}$ with three extra points (for the three strata~$\cV_{G/H}$ with~$G/H\simeq C_2$). We recover in this way a space that we displayed in~\cite[\S\,18]{tt-perm}:
\begin{figure}[H]
\centering
\includegraphics[scale=0.2]{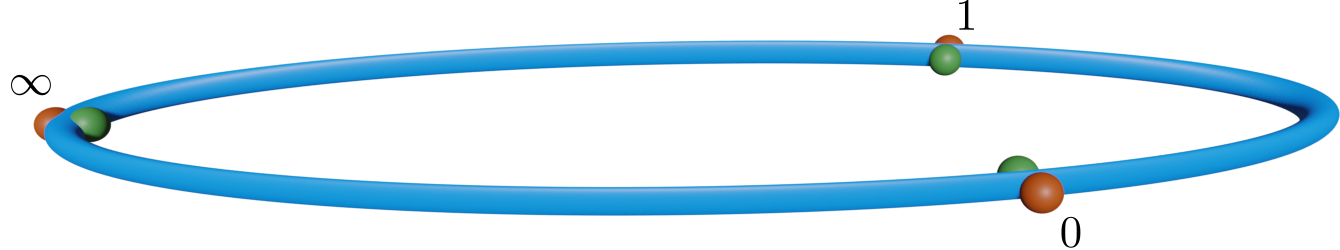}
\caption{Artist rendering of $\Spc(\stperm(V_4;k))$.}
\label{fig:P1-doubled}
\end{figure}
\end{Exa}

\begin{Rem}
More generally, for any elementary abelian $p$-group~$G$, the strata~$\cV_{G/H}$ in~\eqref{eq:spcK-strata} for $H\le G$ are extended projective spaces~$\bar{\mathbb{P}}^{n-1}_{\!\!k}$, where $n$ is the $p$-rank of~$G/H$, see~\cite[Example~8.6]{tt-perm}.
In fact, the space~$\Spc(\DPerm(G;k)^c)$ has a natural structure of a Dirac scheme, see~\cite[Corollary~15.6]{tt-perm}.
\end{Rem}

\begin{Rem}
The punctured spectrum
\begin{align*}
\Spc(\stmod(kG))\ &\subset\ \Spc(\Db(kG))
\intertext{
obtained by removing the closed point is precisely the \emph{periodic locus} of~$\Db(kG)$, namely the open subset of those points~$\cP\in\Spc(\Db(kG))$ such that the local category $\Db(kG)/\cP$ has a periodic suspension: $\Sigma^d\unit\simeq\unit$ for some $d\neq 0$.
Similarly, we expect that the punctured spectrum}
\Spc(\stperm(G;k))\ &\subset\ \Spc(\cK(G))
\end{align*}
obtained by removing the closed points is the periodic locus of~$\cK(G)=\DPerm(G;k)^c$.
This characterization of the stable permutation category~$\stperm(G;k)$ is established for $p$-groups in~\cite[Theorem~8.5]{gallauer:periods}.
For general finite groups it remains conjectural at this stage.
\end{Rem}

\section{General properties}
\label{sec:general}%

Recall that the assignment $G\mapsto \DPerm(G;k)$ comes with restriction functors along any group homomorphisms~$u\colon G\to G'$, that have an ambidextrous adjoint in case $u$ is injective, namely (co)induction. This data forms a Mackey 2-functor in the sense of~\cite{BalmerDellAmbrogio20}.
We claim that the 2-functor~$G\mapsto\StPerm(G;k)$ inherits a structure of Mackey 2-functor but only for restriction along injective homomorphisms.
The restriction and induction functors are compatible with those of~$\DPerm$ under the localizations $\DPerm(G;k)\onto \StPerm(G;k)$.
This is reminiscent of what happens with the passage from the usual derived category to the stable module category, see~\cite[Example~4.2.6]{BalmerDellAmbrogio20}.
In the language of~\cite{BalmerDellAmbrogio20}, this $G\mapsto \StPerm(G;k)$ is a Mackey 2-functor defined on the 2-category of finite groupoids and faithful 1-cells.
The analogous statements hold for the 2-functor $G\mapsto \stperm(G;k)$ of subcategories of compacts.

\begin{Prop}
\label{Prop:2-functoriality}%
The \eqperf\ complexes of $p$-permutation modules (\Cref{Def:G-perfect}) are preserved by conjugation by elements in~$G$, by restriction to subgroups and by induction from subgroups. They are also preserved by modular fixed-points functors with respect to every $p$-subgroup.
\end{Prop}
\begin{proof}
For restriction $\Res^G_K$ to a subgroup~$K\le G$, we can use its compatibility with modular fixed points (\cite[Proposition~5.15]{tt-perm}): for every $p$-subgroup~$H\le K$ we have $\check\Psi^H\circ \Res^G_K\cong \Res^{\WGH}_{\Weyl{K}{H}}\circ \check\Psi^{H}$. Since $\Weyl{K}{H}$ is a subgroup of~$\WGH$, the restriction of a perfect complex over~$\WGH$ remains perfect over~$\Weyl{K}{H}$.
\cite[Proposition~5.15]{tt-perm} also deals with conjugation, in the same way.

For modular fixed-point, the argument is similar, using \cite[Proposition~5.17]{tt-perm} instead. Note that there is also a restriction to a subgroup involved in that case.

Finally, for induction, one can argue by means of supports, using \Cref{Prop:supp(G-perf)}. It suffices to check that, for every $K\le G$ and every $C\in\cK(K)$, the support of the induced complex $\Ind_K^G(C)$ is equal to the image of the support of~$C$ under~$\Spc(\Res^G_K)$. This is a general tt-fact about the image of a rigid object by a right adjoint to a separable extension~\cite[Theorem~3.4\,(c)]{Balmer16b}. It then suffices to use that $\Spc(\Res^G_K)$ maps closed points to closed points~\cite[Lemma~11.9\,(b)]{tt-perm}.
\end{proof}

\begin{Cor}
\label{Cor:2-functoriality}%
Let $H\le G$ be a subgroup. There is a well-defined coproduct-preserving tt-functor $\Res^G_H\colon \StPerm(G;k)\to \StPerm(H;k)$ compatible with the functor $\Res^G_H\colon \DPerm(G;k)\to \DPerm(H;k)$ under localization, meaning that the following diagram commutes up to isomorphism
\[
\xymatrix{
\DPerm(G;k) \ar@{->>}[d]_-{p} \ar[r]^-{\Res^G_H}
& \DPerm(H;k) \ar@{->>}[d]^-{p}
\\
\StPerm(G;k) \ar[r]^-{\Res^G_H}
& \StPerm(H;k).\!
}
\]

Restriction admits a two-sided adjoint $\Ind_H^G\colon \StPerm(H;k)\to \StPerm(G;k)$, compatible with $\Ind_H^G\colon \DPerm(H;k)\to \DPerm(G;k)$ under localization (as above).

Let $H\le G$ be a $p$-subgroup. There exists a tt-functor of `stable' modular $H$-fixed-points $\Psi^H\colon \StPerm(G;k)\to \StPerm(\WGH;k)$, compatible with the original~$\Psi^H\colon \DPerm(G;k)\to \DPerm(\WGH;k)$ under localization (as above again).

All the above functors preserve compact objects and give similar statements for~$\stperm$ instead of~$\StPerm$, \textsl{mutatis mutandis}.
\end{Cor}
\begin{proof}
Once the original functors preserve the tt-ideals of compacts that we quotient out (\eqperf\ complexes), these functors pass to the finite localizations of the `big' categories in a unique way. Natural transformations are preserved, and  therefore adjunctions follow. Since the original functors preserve compacts, so do the new ones on stable categories.
\end{proof}
\begin{Cor}
\label{Cor:2-Mackey}%
The 2-functors~$G\mapsto \StPerm(G;k)$ and~$G\mapsto \stperm(G;k)$ (with respect to restriction) are Mackey 2-functors on the 2-category of finite group(oid)s and \emph{faithful} 1-cells.
\end{Cor}
\begin{proof}
All properties, in particular Ambidexterity and the Mackey formula~\cite[Definition~1.1.7]{BalmerDellAmbrogio20} are inherited by localization since all units and counits come from the Mackey 2-functor~$\DPerm(-;k)$, as long as we only use restriction under faithful homomorphisms as in \Cref{Cor:2-functoriality}. We leave the details to the reader.
\end{proof}

\begin{Cor}
For every subgroup~$H\le G$ the category $\StPerm(H;k)$ is equivalent to the category of $A^G_H$-modules in~$\StPerm(G;k)$ where the tt-ring $A^G_H=\Ind_H^G(\unit)$ is~$k(G/H)$ with the usual multiplication coming from~$\perm(G;k)$ (making all~$\gamma\in G/H$ orthogonal idempotents). The same holds on subcategories of compacts.
\end{Cor}
\begin{proof}
Again, this passes under localization or holds by general 2-Mackey theory~\cite[Theorem~2.4.1]{BalmerDellAmbrogio20}. One obtains~$\Ind_H^G(\unit)=k(G/H)$ from~$\DPerm$.
\end{proof}

\begin{Rem}
\label{Rem:stperm-cohomological}%
One can actually show that the Mackey 2-functor~$\StPerm(-;k)$ is \emph{cohomological} in the sense of~\cite{BalmerDellAmbrogio24}, meaning that for every subgroup~$H\le G$ the composite $\Id\to \Ind_H^G \Res^G \to \Id$, of the unit for the $\Res\adj\Ind$ adjunction with the counit for $\Ind\adj \Res$, is equal to multiplication by~$[G\!:\!H]$. Again, this is immediate from the fact that the original $\DPerm(-)$ is itself cohomological and from the fact that the units and counits pass to the localization.
\end{Rem}
\begin{Rem}
\label{Rem:descent}%
It follows from \Cref{Rem:stperm-cohomological} that restriction $\Res^G_H\colon \StPerm(G;k)\to \StPerm(H;k)$ is a faithful functor whenever the index~$[G\!:\!H]$ is prime to~$p$. Hence this restriction functor satisfies effective descent, in the classical sense. More precisely, the canonical Eilenberg-Moore functor $\cT=\StPerm(G;k)\to \Desc_{\cT}(A)$ to the descent category for~$A=A^G_H$ in~$\cT$ is an equivalence by~\cite[Corollary~3.1]{Balmer12}.
\end{Rem}

\begin{Rem}
In the same vein, one can prove that the Mackey 2-functor $\StPerm$ inherits from~$\DPerm$ the property of being a \emph{Green} 2-functor, in the sense of Dell'Ambrogio~\cite{DellAmbrogio22}.
This follows from the fact that $\DPerm\onto \StPerm$ is not only a quotient by a localizing subcategory but by a tensor-ideal.
\end{Rem}

We can adapt the Colimit Theorem from the case of the derived permutation category~\cite{tt-perm} to the stable permutation category.
\begin{Rec}
\label{Rec:EpG}%
The category $\mathcal{E}_{p}(G)$ of elementary abelian subquotients of~$G$ is defined as follows. Its objects are pairs $(H,K)$ where $K\normaleq H$ is a section of~$G$ and $H/K$ is elementary abelian.
A morphism $(H,K)\to (H',K')$ is an element~$g\in G$ such that $K'\le K^g$ and~$H^g\le H'$. By \cite[Construction~11.4]{tt-perm} each such morphism $g\colon (H,K)\to (H',K')$ yields a tt-functor
\begin{equation}\label{eq:K(g)}%
\cK(g)\colon \cK(H'/K')\xto{\Psi^{\bar K}}\cK(H'/K^g)\xto{\Res}\cK(H^g/K^g)\isoto \cK(H/K)
\end{equation}
combining modular fixed points with respect to the $p$-subgroup $\bar{K}:=K^g/K'\normal H'/K'$, restriction to the subgroup~$H^g/K^g\le H'/K^g$ and conjugation by~$g$.
Applying $\Spc$ gives a continuous map~$\Spc(\cK(H/K))\to \Spc(\cK(H'/K'))$ and the colimit of this diagram of topological spaces is~$\Spc(\cK(G))$ by~\cite[Theorem~11.10]{tt-perm}:
\begin{equation}
\label{eq:colim-K}%
\colim_{(H,K)\in\mathcal{E}_{p}(G)}\Spc(\cK(H/K))\isoto\Spc(\cK(G)).
\end{equation}
The comparison maps from each $\Spc(\cK(H/K))$ to~$\Spc(\cK(G))$ are similarly induced by the tt-functor $\cK(G)\xto{\Psi^K}\cK(\WGK)\xto{\Res}\cK(H/K)$.
\end{Rec}

We now prove the analogous statement for the stable permutation category.
\begin{Thm}
\label{Thm:colim-stperm}%
Modular fixed points and restrictions yield as above a homeomorphism of topological spaces
\begin{equation}
\label{eq:colim-stperm}%
\colim_{(H,K)\in\mathcal{E}_{p}(G)}\Spc(\stperm(H/K;k))\isoto\Spc(\stperm(G;k)).
\end{equation}
\end{Thm}
\begin{proof}
In view of \Cref{Cor:Spc(stperm)}, it suffices to verify that for each of the various functors $f^*$ discussed above (namely modular fixed points, restriction and conjugation) with induced map $f=\Spc(f^*)$ on spectra, we have $f^{-1}(\{\text{closed points}\})=\{\text{closed points}\}$.
For then the homeomorphism~\eqref{eq:colim-stperm} is the restriction of the homeomorphism~\eqref{eq:colim-K} to the open complement of the closed points.

The inclusion $\subseteq$ follows from \Cref{Prop:supp(G-perf),Prop:2-functoriality}.
The converse inclusion is~\cite[Lemma~11.9\,(b)]{tt-perm}.
\end{proof}

\section{Bottleneck subgroups}
\label{sec:bottleneck}%

In this short preparatory section, $P$ is a $p$-group.
We discuss when $P$ admits a `bottleneck' subgroup~$H$ by which we mean a proper non-trivial subgroup~$1\neq H<P$ that cannot be surrounded by a section $K'<H<H'$ with $K'\normaleq H'$ and~$H'/K'$ elementary abelian of $p$-rank at least two. In other words, we discuss when every subgroup~$H'$ containing~$H$ as a subgroup of index~$p$ has only $H$ as index-$p$ subgroup.

Bottleneck subgroups can exist at the top and at the bottom of a $p$-group:
\begin{Rem}
\label{Rem:unique-index-p}%
It is well-known that a $p$-group that contains a \emph{unique index-$p$ subgroup}~$H$ must be cyclic. Indeed, $H$ must contain any central element of order~$p$ and, modding it out, one concludes by induction on the order of the group.
\end{Rem}
\begin{Rem}
\label{Rem:quaternion}%
At the other end, a $p$-group that has a \emph{unique subgroup of order~$p$} must be either a cyclic group or a generalized quaternion group $Q_{2^n}$ of order~$2^n$.
This is more involved, see~\cite[Theorem~4.3, p.~99]{brown:cohomology-groups}.
Recall that for $n\ge3$, the \emph{generalized quaternion group} is presented as follows
\begin{equation}
\label{eq:Q2n}%
Q_{2^n}=\langle x,y\mid x^{2^{n-1}}=1,\ y^2=x^{2^{n-2}}, y x y\inv=x\inv \rangle.
\end{equation}
Its unique subgroup of order~2 is the center~$Z(Q_{2^n})=\langle y^2\rangle$.

Note that in those examples, both in the cyclic case and in the generalized quaternion case, the unique subgroup of order~$p$ is central.
\end{Rem}

We are going to show that these are the only examples.
\begin{Lem}
\label{Lem:bottleneck}%
Let $P$ be a $p$-group that contains a non-trivial proper subgroup~$H$ with the following property: For every subgroup $H'\le P$ containing~$H$ with~$[H'\!:\!H]=p$, the group~$H'$ is cyclic. Then $P$ is cyclic or generalized quaternion.
\end{Lem}
\begin{proof}
By induction, we can assume the result for every $p$-group of order~$|P|/p$.
Recall that for every proper subgroup~$H$ of a $p$-group~$P$, there exists $H<H'\le P$ with $[H'\!:\!H]=p$. It follows from our assumption ($H'$ being cyclic) that $H$ is cyclic.

\smallbreak
\textit{Observation:
Let $z\in Z(P)$ be a central element of order~$p$. Then $z\in H$.}
\smallbreak

Indeed, if $z$ was not in~$H$ then the subgroup $H'=\langle H,z\rangle\simeq H\times C_p$ would contradict our hypothesis. Let us pick such a~$z\in Z(P)\cap H$.

We distinguish two cases. First, suppose that $H=\langle z\rangle$. Then $H\le Z(P)$ is central. For every element $g\in P$ of order~$p$, the subgroup~$H':=\langle z,g\rangle$ cannot be bigger than~$H$ otherwise it would be isomorphic to~$C_p\times C_p$ and contradict our hypothesis again. It follows that $g\in H$. Consequently, $H$ is the only subgroup of~$P$ of order~$p$. By \Cref{Rem:quaternion} we obtain that $P$ is cyclic or generalized quaternion.

The second option is that $H$ is strictly larger than the central cyclic subgroup~$\langle z\rangle$.
Consider then the quotient~$\bar{P}=P/\langle z\rangle$ and $\bar{H}=H/\langle z\rangle\lneq \bar P$.
One easily verifies that $\bar H$ still satisfies the same hypothesis in~$\bar P$ as~$H$ did in~$P$.
By induction hypothesis, we know that $\bar P$ is either cyclic or generalized quaternion, and in particular admits a unique subgroup of order~$p$, which is central and therefore necessarily in~$\bar{H}$ by the above Observation (applied to~$\bar{P}$ and~$\bar{H}$).

Let $g\in P$ be an element of order~$p$. We claim that $g\in H$. Its class~$\bar{g}\in\bar{P}$ modulo~$\langle z\rangle$ has order~$1$ or~$p$ hence belongs to~$\bar{H}$ by the above discussion. It follows from this and from~$\ideal{z}\le H$ that $g\in H$ as claimed.

Consequently, $H$ contains every subgroup of~$P$ of order~$p$ and since the cyclic group $H$ has a unique subgroup of order~$p$ the same holds for~$P$.
We conclude that $P$ must be cyclic or generalized quaternion by \Cref{Rem:quaternion}.
\end{proof}

We shall use the above result in the following form:

\begin{Prop}
\label{Prop:bottleneck}%
Let $P$ be a $p$-group that is neither cyclic nor generalized quaternion and let $1\neq H<P$ be a proper non-trivial subgroup.
Then there exist a section~$K'\normaleq H'\le P$ and proper inclusions $K'<H<H'$, with $H'/K'$ elementary abelian (necessarily of $p$-rank at least two).
\end{Prop}
\begin{proof}
By \Cref{Lem:bottleneck} there exists a \emph{non-cyclic} subgroup~$H'>H$ with $[H'\!:\!H]=p$. Since $H'$ is not cyclic, it has more than one subgroup of index~$p$ (\Cref{Rem:unique-index-p}). We get the result with $K'=\bigcap_{[H'\,:\,N]=p}N$ the Frattini subgroup of~$H'$.
\end{proof}


\section{Indecomposability}
\label{sec:indecomposability}%

We can apply \Cref{Thm:colim-stperm} and \Cref{Prop:bottleneck} to prove a first result about the stable permutation category, namely deciding when it decomposes.

\begin{Thm}
\label{Thm:indecomposability}%
Let $G$ be a group, whose $p$-Sylow is neither cyclic, nor generalized quaternion (for $p=2$). Then $\stperm(G;k)$ is indecomposable as a tensor-triangulated category, that is, its spectrum is connected.
\end{Thm}
\begin{proof}
The fact that $\Spc(\cK)$ disconnected implies $\cK\simeq \cK_1\times \cK_2$ when $\cK$ is rigid is direct from~\cite[Theorem~2.11]{Balmer07}.
Let $P$ be a $p$-Sylow of~$G$. Since the tt-functor $\Res^G_P\colon \stperm(G;k)\to \stperm(P;k)$ is faithful by \Cref{Rem:stperm-cohomological}, the continuous map $\Spc(\Res^G_P)$ is surjective (\cite{Balmer18a}) and we can therefore reduce the theorem to the case where $G=P$ is a $p$-group.

By \Cref{Thm:colim-stperm}, the space $\Spc(\stperm(G))$ is covered by the images of the spaces $\Spc(\stperm(H/K))$ for $(H,K)\in\mathcal{E}_{p}(G)$ in the category of elementary abelian subquotients of~$G$ recalled in~\Cref{Rec:EpG}. These spaces are non-empty as long as~$H/K$ is non-trivial.

We know from \cite[Proposition~15.11]{tt-perm} that for every non-trivial elementary abelian $p$-group~$E$ (for instance any~$H/K$ above) the spectrum $\Spc(\Kb(\perm(E;k))$ is irreducible: it has a generic point.
The closed points in~$\Spc(\Kb(\perm(E;k))$ are of height equal to the $p$-rank~\cite[Remark~15.15]{tt-perm}. Hence, after removing them, the open $\Spc(\stperm(E;k))$ remains irreducible and in particular connected. It follows that any continuous image of that space $\Spc(\stperm(E;k))$ remains connected.
We are going to see that all those images inside~$\Spc(\stperm(G;k))$ are connected by showing they overlap gradually.

Let us look at the `top' elementary abelian section, namely~$(G,F)$ where $F\normal G$ is the Frattini subgroup of our~$p$-group~$G$.
Let us denote by $A\subseteq \Spc(\stperm(G;k))$ the connected component that contains the image of $\Spc(\stperm(G/F;k))$.
We want to show that $A$ contains every other image of~$\Spc(\stperm(H/K;k))$ for every other non-trivial section~$(H,K)\in\mathcal{E}_p(G)$.
We proceed by induction on~$[G\!:H]$. As long as $H=G$, the image of $\Spc(\stperm(H/K;k))$ is contained in~$A$ since the Frattini is the smallest: $F\le K$ gives $\Img(\Spc(\Psi^K))\subseteq\Img(\Spc(\Psi^F))$. So we can assume $[G\!:\!H]>1$ and that the image of every $(H',K')$ with $[G\!:\!H']<[G\!:\!H]$ is contained in~$A$.

Let $(H,K)\in\mathcal{E}_p(G)$ with $H/K$ non-trivial and $H\lneq G$ proper.
By \Cref{Prop:bottleneck} there exists $(H',K')\in\mathcal{E}_p(G)$ with strict inclusions~$K'<H<H'$.
Let $F'<H$ be the Frattini subgroup of~$H$.
We have the following left-hand side diagram of subgroups of~$G$, in which a line indicates a normal subgroup with elementary abelian quotient:
%
\[
\vcenter{\xymatrix@R=.5em{
&& H' \ar@{-}[ld]_-{} \ar@{-}[dd]^-{}
\\
& H \ar@{-}[ld]_-{} \ar@{-}[rd]^-{}  \ar@{-}[dd]^-{}
\\
K \ar@{-}[rd]_-{}
&& K' \ar@{-}[ld]_-{}
\\
& F'
}}
\kern 5em
\vcenter{\xymatrix@R=2em{
(H,K) \ar[r]
& (H,F')
\\
(H,K') \ar[ru] \ar[r]
& (H',K')
}}
\]
Those inclusions yield the morphisms (all given by the identity $1\in G$) in~$\mathcal{E}_{p}(G)$ displayed on the right-hand side above. Since all $E\in\{H/K,\ H/F',\ H/K',\ H'/K'\}$ are non-trivial, the spectra $\Spc(\stperm(E;k))$ are non-empty.
And since they are all connected, we conclude that the image of~$\Spc(\stperm(H/K;k))$ for the section~$(H,K)$ belongs to the same connected component as that of~$(H',K')$. Since $H'$ is larger than~$H$, we conclude by induction hypothesis that this connected component is indeed~$A$.
\end{proof}

\section{Decomposition over cyclic and generalized quaternion}
\label{sec:cyclic-quaternion}%

We want to prove a converse to \Cref{Thm:indecomposability}. If the $p$-Sylow of~$G$ is cyclic or generalized quaternion (for $p=2$ in the latter case) then $\StPerm(G;k)$ will decompose and we want to identify the components.

We are going to use a general fact about `big' tt-categories~$\cT$, \ie compactly-rigidly generated tensor-triangulated categories. Recall from \Cref{Rem:K|U} that $\cT|_U:=\cT/\Loc{\cT^c_Y}$ when $U=Y^\compl$ is the complement of a Thomason subset (\eg\ a quasi-compact open).
Any coproduct-preserving tt-functor~$f^*\colon \cT\to \cS$ between big tt-categories induces a continuous map~$f:=\Spc(f^*)\colon \SpcS\to \SpcT$ and the functor $f^*$ sends $\cT_Y$ into~$\cS_{f\inv(Y)}$ by~\cite[Theorem~6.3]{balmer-favi:idempotents}. In particular $f\inv(U)^\compl=f\inv(U^\compl)$ is Thomason in$~\SpcS$.
It follows that $f^*$ induces a tt-functor on the localizations~$\bar f^*\colon \cT|_U\to \cS|_{f\inv(U)}$, that we can denote~$f^*|_U$.

\begin{Prop}
\label{Prop:ff-equiv}%
With above notation for~$f^*\colon \cT\to \cS$ and~$U\subseteq\SpcT$, suppose that $f^*$ is fully faithful. Then the induced functor
\[
f^*|_U\colon \cT|_U \to \cS|_{f\inv(U)}
\]
remains fully faithful. Moreover, suppose in addition that $\cS$ is generated as a tensor-triangulated category by~$f^*(\cT)$ and some collection of objects supported outside of~$f\inv(U)$, that is, objects in~$\Loc{\cS^c_{f\inv(U)^\compl}}$. Then $f^*|_U$ is an equivalence.
\end{Prop}
\begin{proof}
The full-faithfulness of~$f^*|_U$ is~\cite[Lemma~5.1]{BalmerGallauer25a}. For the moreover part, the generators supported outside of~$f\inv(U)$ become zero in~$\cS|_{f\inv(U)}$ by definition. It follows that $f^*|_U$ is also essentially surjective.
\end{proof}

\begin{Rem}
We are going to apply the above to $f^*$ given by inflation along $G\onto G/N$ for a particular normal subgroup~$N\normaleq G$. Note that inflation does not descend to stable permutation categories, very much for the same reasons as with the usual stable module category: It does not preserve \eqperf\ complexes. (For instance $k(G/N)$ is inflated from a $G/N$-perfect but if $N$ admits a non-trivial $p$-subgroup~$H\le N$ the object~$\Psi^H(k(G/N))\simeq k^{[G:N]}$ is not perfect.) However, the version of inflation that we shall use is the one on derived permutation categories $\Infl^{G/N}_G\colon \DPerm(G/N;k)\to \DPerm(G;k)$, which is fine~\cite[\S\,4]{tt-perm}.
\end{Rem}

\begin{Prop}
\label{Prop:Infl-ff}%
For every normal subgroup $N\normaleq G$, the coproduct-preserving tt-functor $\Infl^{G/N}_G\colon \DPerm(G/N;k)\to \DPerm(G;k)$ is fully faithful.
\end{Prop}
\begin{proof}
As with any coproduct-preserving tt-functor one can verify full-faithfulness on compacts (check that the unit $\Id\to f_*f^*$ is an isomorphism and use cocontinuity of~$f^*$ and~$f_*$). Here, the functor $\Infl^{G/N}_G\colon \Kb(\perm(G/N;k)^\natural)\to \Kb(\perm(G;k)^\natural)$ is fully faithful because so is $\Infl^{G/N}_G\colon \perm(G/N;k)\to\perm(G;k)$, the latter because it is the restriction of the usual restriction-of-scalars functor on modules $\Infl^{G/N}_G\colon \mmod{k(G/N)}\to\mmod{kG}$ along a ring epimorphism~$kG\onto k(G/N)$.
\end{proof}

Let us start discussing our examples of cyclic or generalized quaternions groups.
We begin with the derived category of permutation~$\cK(G)=\DPerm(G;k)^c=\Kb(\perm(G;k)^\natural)$ as recalled in \Cref{Rec:perm,Rec:tt-perm}.
\begin{Lem}
\label{Lem:disconnect}%
Let $G$ be a $p$-group with a unique subgroup of order~$p$ (\Cref{Rem:quaternion}). Let~$N=Z(G)\normal~G$ be that subgroup.
\begin{enumerate}[\rm(a)]
\item
\label{it:disconnect-1}%
The spectrum $\Spc(\cK(G))$ is the union of two closed subsets~$Z_1=\Img(\Spc(\Psi^N))$ and $Z_2=\supp(k(G/N))$ whose intersection is exactly the closed point~$\cM(N)$.
\smallbreak
\item
\label{it:disconnect-2}%
For the stable permutation category, the open subspace~$\Spc(\stperm(G;k))\subset\Spc(\cK(G))$ of \Cref{Cor:Spc(stperm)} is disconnected as follows
\[
\Spc(\stperm(G;k))=\Img(\Spc(\Psi^N))\ \sqcup\ \{*\}
\]
where $\Psi^N\colon \stperm(G;k)\to \stperm(G/N;k)$ is (stable) modular $N$-fixed points and the singleton $\{*\}$ is the image of $\Spc(\stmod(G;k))=\{*\}$.
\smallbreak
\item
\label{it:disconnect-3}%
There is an equivalence of tt-categories
\[
\StPerm(G;k)\isoto \StPerm(G/N;k)\times \StMod(G;k)
\]
given by modular $N$-fixed-points~$\Psi^N\colon \StPerm(G;k)\to \StPerm(G/N;k)$ and the canonical localization~$\bar\Upsilon\colon \StPerm(G;k)\onto \StMod(G;k)$.
This decomposition preserves the compact objects and induces~\eqref{it:disconnect-2} on spectra.
\end{enumerate}
\end{Lem}
\begin{proof}
Since $N$ is normal the map $\psi^N:=\Spc(\Psi^N)\colon \Spc(\cK(G/N))\to \Spc(\cK(G))$ is a homeomorphism onto its closed image (\cite[Proposition~7.18]{tt-perm}).
For every non-trivial subgroup~$H\le G$ we have $N\normaleq H$ and therefore the $H$-fixed-points functor $\Psi^H$ factors via the $N$-fixed-points functor~$\Psi^N$ (\cite[Corollary~5.18]{tt-perm}). Hence $Z_1=\Img(\Psi^N)$ contains $\Img(\Psi^H)$ and in particular the stratum~$\cV_{\WGH}$ in~\eqref{eq:spcK-strata}.
Thus the only stratum of~$\Spc(\cK(G))$ in~\eqref{eq:spcK-strata} that is not contained in~$Z_1$ is the cohomological open~$\cV_{\Weyl{G}{1}}=\Spc(\Db(kG))$. Since $N$ is the maximal elementary abelian subgroup of~$G$, or by direct computation of the cohomology, we see that $\Spc(\Res^G_N)$ yields a bijection of Sierpi\'nski 2-point spaces $\Spc(\Db(kN))\isoto \Spc(\Db(kG))$. Since $Z_2=\supp(k(G/N))$ is as usual the image of $\Spc(\Res^G_N)$ it follows that $Z_2$ contains the last stratum~$\cV_G$ and therefore $Z_1\cup Z_2$ is indeed the whole of~$\Spc(\cK(G))$ by~\eqref{eq:spcK-strata}. Their intersection $Z_1\cap Z_2=\Img(\psi^H)\cap Z_2$ can be computed by looking at the preimage of~$Z_2=\supp(k(G/N))$ under~$\psi^N$, which is simply the support of~$\Psi^N(k(G/N))\cong k(G/N)$ in~$\cK(G/N)$. That free $k(G/N)$-module has support~$\cM_{G/N}(1)$ in~$\Spc(\cK(G/N))$. Its image under~$\psi^N\colon \Spc(\cK(G/N)\isoto Z_1$ is~$\psi^H(\cM_{G/N}(1))=\cM_G(N)$. This proves~\eqref{it:disconnect-1}.

By \Cref{Cor:Spc(stperm)}, we can intersect the closed cover $\Spc(\cK(G))=Z_1\cup Z_2$ with the open subset $\Spc(\stperm(G))$ of~$\Spc(\cK(G))$ and we get the disconnection:
\[
\Spc(\stperm(G;k))=U_1\sqcup U_2\quad\textrm{where}\quad U_i=Z_i\cap \Spc(\stperm(G;k)),\ i=1,2
\]
since $Z_1\cap Z_2$ is one of the closed points we remove in this process. It follows immediately from \Cref{Cor:2-functoriality} that $U_1=\Img(\psi^N)\cap \Spc(\stperm(G;k))$ remains the image of $\Spc(\Psi^N)$ for the functor~$\Psi^N$ descended to the stable permutation categories. As we saw above, $Z_2$ consisted of 3 points, $\cM(N)$ and the two points of~$\Spc(\Db(kG))$, including~$\cM(1)$. Removing the two closed ones,~$\cM(N)$ and~$\cM(1)$, we are left with $U_2=\{\ast\}$ the singleton from~$\Spc(\stmod(kG))$.
This gives~\eqref{it:disconnect-2}.

By~\eqref{it:disconnect-2}, the tt-category~$\stperm(G;k)$ must decompose as $\cK_1\times \cK_2$ with $\Spc(\cK_1)=U_1=\Img(\psi^N)$ and~$\Spc(\cK_2)=U_2=\{*\}$; see~\cite{Balmer07}. It is easy to identify the second component: It is the restriction of~$\cK(G)$ on the open singleton~$\{*\}=\Spc(\stmod(kG))$ in~$\Spc(\cK(G))$. By \Cref{Prop:stperm->>stmod} we know that this quotient is~$\stmod(kG)$. The main claim is that $\Psi^N\colon \stperm(G;k)\to \stperm(G/N;k)$ induces an equivalence $\cK_1\isoto \stperm(G/N;k)$ on~$U_1$. First note that this exists for the closed complement of~$U_1$, namely the singleton~$\{*\}$, is the support of~$k(G/N)$ as discussed in the first part, and $\Psi^N(k(G/N))=k(G/N)$ becomes zero in~$\stperm(G/N;k)$ by \Cref{Exa:naive-G-perfect}.
Furthermore, since the composite $\Psi^N\circ\Infl^{G/N}_G\colon \cK(G/N)\to \cK(G)\to \cK(G/N)$ is isomorphic to the identity, it suffices to show that $\Infl^{G/N}_G$ descends to an equivalence $\stperm(G/N;k)\isoto \cK_1$ to get the result. We need to be careful since inflation does not exist on the entire stable permutation categories.

We use the big tt-categories.
Let us consider $f^*=\Infl^{G/N}_G\colon \DPerm(G/N;k)\to \DPerm(G;k)$ and $Y:=\{\cM_{G/N}(1)\}$ the closed point in the cohomological open of~$G/N$, which is the support of the free $k(G/N)$-module~$k(G/N)$. Again, we write $f=\Spc(f^*)\colon \Spc(\cK(G))\onto \Spc(\cK(G/N))$ for the induced map on spectra. The preimage $f\inv(Y)$ in~$\Spc(\cK(G))$ of the subset~$Y=\supp(k(G/N))$ is as usual the support of the image object~$\Infl^{G/N}_G(k(G/N))$ by the functor~$f^*$, namely the permutation $kG$-module~$k(G/N)$. The support of the latter is the closed subset~$Z_2$ of~\eqref{it:disconnect-1}.
By fully-faithfulness of~$f^*$ (\Cref{Prop:Infl-ff}), we can apply \Cref{Prop:ff-equiv} to get a well-defined fully-faithful functor
\begin{equation}
\label{eq:aux-f*}%
f^*|_{Y^\compl}\colon \DPerm(G/N;k)|_{Y^\compl}\to \DPerm(G;k)|_{Z_2^\compl}
\end{equation}
from the localization of~$\cT=\DPerm(G/N;k)$ on the open~$Y^\compl=\Spc(\cK(G/N))\sminus\{\cM_{G/N}(1)\}$ to the localization of~$\cS=\DPerm(G;k)$ on the open~$Z_2^\compl=Z_1\sminus\{\cM(N)\}$ of~$\Spc(\cK(G))$, where we use~\eqref{it:disconnect-1} in the last equality.
We claim that this $f^*|_{Y^\compl}$ is an equivalence by the `moreover part' of~\Cref{Prop:ff-equiv}. Indeed, the generators $\SET{k(G/H)}{H\le G}$ of~$\DPerm(G;k)$ are almost all in the essential image of~$f^*$, that is, inflated from~$G/N$. This is because $N$ is contained in all non-trivial subgroups of~$G$. Only the generator $kG$ is not in the essential image of~$f^*$. But $kG$ has support in~$Z_2$. So~\eqref{eq:aux-f*} is indeed an equivalence by \Cref{Prop:ff-equiv}.

Note that the two tt-categories in~\eqref{eq:aux-f*} have (homeomorphic) spectra consisting of~$Y^\compl=\Spc(\cK(G/N))\sminus\{\cM_{G/N}(1)\}$ and $Z_2^\compl=\Spc(\cK(G))\sminus \supp(k(G/N))$ respectively, which are homeomorphic under the restriction of~$\psi^H=\Spc(\Psi^H)$ on~$Y^\compl$, with inverse $\Spc(f^*)|_{Z_2^\compl}$.
We can further localize these two categories by removing all the remaining closed points~$\SET{\cM_{G/N}(\bar H)}{1\neq\bar{H}\le G/N}$ in the former, respectively their images $\SET{\cM_{G}(H)}{N\lneq H\le G}$ in the latter. In other words, on the left-hand side of~\eqref{eq:aux-f*} we localize $\DPerm(G/N;k)$ away from all its closed points, which gives the stable permutation category of~$G/N$, and on the right-hand side of~\eqref{eq:aux-f*} we localize away from all the closed points and away from~$Z_2$, which gives the stable permutation category of~$G$ localized away from (what remains of)~$Z_2$. This is precisely the wanted localization $\StPerm(G;k)|_{U_1}$ since $U_1=\Spc(\stperm(G;k))\cap Z_2^\compl$ by construction. Consequently $f^*=\Infl^{G/N}_G$ yields a tt-equivalence
\[
\StPerm(G/N;k)\isoto \StPerm(G;k)|_{U_1}.
\]
Since $\Psi^N$ is its left inverse, it is also an equivalence and we get~\eqref{it:disconnect-3}.
\end{proof}

\begin{Thm}
\label{Thm:quaternion}%
Let $p=2$ and let $G$ be a finite group whose $2$-Sylow is generalized quaternion of order~$2^n$, as in~\eqref{eq:Q2n}.
Let $H\le G$ be a cyclic subgroup of order~$2$, \ie the center of some $2$-Sylow of~$G$.
Then (stable) modular $H$-fixed-points $\Psi^{H}$ (\Cref{Cor:2-functoriality}) and the canonical localization~$\bar\Upsilon$ in~\eqref{eq:four-tt-cats} yield an equivalence
\[
(\Psi^{N},\bar\Upsilon)\colon \StPerm(G;k)\isoto \StPerm(\WGH;k)\times \StMod(k(G)))
\]
corresponding to a decomposition $\Spc(\stperm(G;k))=\Img(\Spc(\Psi^H))\sqcup\{*\}$.
\end{Thm}
\begin{proof}
Let $P<G$ be a $2$-Sylow of~$G$ containing~$H$, necessarily as its center (\Cref{Rem:quaternion}).
We have $P/H\le \WGH$ and the following diagram commutes (as follows from~\cite[Proposition~5.15]{tt-perm} by localization):
\begin{equation}
\label{eq:aux-quaternion}%
\vcenter{\xymatrix@C=5em{
\StPerm(G;k)\ar[r]^-{(\Psi^{H},\bar\Upsilon)} \ar[d]_-{\Res^G_P}
& \StPerm(\WGH;k)\times \StMod(k(G)) \ar[d]^-{\Res^{\WGH}_{P/H}\times\Res^G_P}
\\
\StPerm(P;k)\ar[r]^-{(\Psi^{H},\bar\Upsilon)}_-{\cong}
& \StPerm(P/H;k)\times \StMod(k(P))).
}}
\end{equation}
The bottom horizontal functor is an equivalence by \Cref{Lem:disconnect} applied to the $p$-group~$P$, and to~$N=Z(P)=H$.
The vertical functors satisfy descent (\Cref{Rem:descent}), identifying~$\cT=\StPerm(G;k)$ with $\Desc_{\cT}(A)$ for~$A=k(G/P)$ and $\cS=\StPerm(\WGH;k)\times \StMod(kG)$ with~$\Desc_{\cS}(B)$ for~$B=\big(k((\WGH)/(P/H)),k(G/P)\big)$ in~$\cS$. The latter tt-ring~$B$ is obtained by applying the adjoint $\Ind_{P/H}^{\WGH}\times\Ind_P^G$ to the unit~$(k,k)$. The first component of~$B$ is~$k(N_G(H)/P)$ since $(\WGH)/(P/H)\cong N_G(H)/P$.
Now the functor~$\Psi^H\colon \perm(G;k)\to \perm(\WGH;k)$ maps~$k(G/P)$ to~$k\big((G/P)^H\big)=k(N_G(H,P)/P)$ and crucially $N_G(H,P)=\SET{g\in G}{{}^gH\le P}$ is equal to~$N_G(H)$ since $H$ is the \emph{unique} cyclic subgroup of order~2 in the generalized quaternion group~$P$. It follows that the top functor in~\eqref{eq:aux-quaternion} maps~$A=k(G/P)$ to~$(\Psi^H(k(G/P)),\bar\Upsilon(k(G/P)))=(k(N_G(H)/P),k(G/P))=B$ and therefore, by general nonsense, it is also an equivalence. (If a tensor functor $F\colon \cT\to \cS$ sends a ring~$A$ to~$F(A)=B$ and if the induced functor $F\colon \Mod_{\cT}(A)\to \Mod_{\cS}(B)$ is an equivalence then $F$ induces an equivalence $\Desc_{\cT}(A)\isoto \Desc_{\cS}(B)$; therefore, if $A$ and~$B$ satisfy descent then the original $F\colon \cT\to \cS$ was an equivalence.)
\end{proof}

\begin{Exa}
\label{Exa:DPerm-Q8}%
Let $p=2$ and $G=Q_8$ be the quaternion group, $N=Z(G)\cong C_2$ and~$G/Z(G)=D_4=V_4$ the Klein-four group. As displayed in~\cite[(2.13)]{tt-perm}, the spectrum~$\Spc(\cK(Q_8))$ of the derived permutation category looks as follows:
\begin{equation}\label{eq:Q_8}%
\kern2em\vcenter{\xymatrix@C=0.1em@R=.4em{
{\color{Brown}\overset{}{\bullet}} \ar@{-}@[Brown][rrdd] \ar@{-}@[Brown][rrrrdd] \ar@{-}@[Brown][rrrrrrdd] \ar@{~}@[Brown][rrrrrrrrdd] &&& {\color{Brown}\overset{}{\bullet}} \ar@{-}@[Brown][ldd] \ar@{-}@[Brown][rrrrrrdd]
&& {\color{Brown}\overset{}{\bullet}} \ar@{-}@[Brown][ldd] \ar@{-}@[Brown][rrrrrrdd]
&& {\color{Brown}\overset{}{\bullet}} \ar@{-}@[Brown][ldd] \ar@{-}@[Brown][rrrrrrdd]
&& {\color{Brown}\overset{\cM(N)}{\bullet}} \ar@{~}@[Brown][ldd] \ar@{-}@[Brown][dd] \ar@{-}@[Brown][rrdd] \ar@{-}@[Brown][rrrrdd]  \ar@{-}@[Gray][rrrrrrrrdd]
&&&&&&&&& {\color{OliveGreen}\bullet} \ar@{-}@[OliveGreen][ldd]^-{{\color{OliveGreen}\cV_{Q_8}\,\cong\,\cV_{C_2}}}
\\ \\
&& {\color{Brown}\bullet} \ar@{-}@[Brown][rrrrrrrdd]_-{\color{Brown}\Img(\psi^N)\,\cong\,\Spc(\cK(V_4))\kern5em}
&& {\color{Brown}\bullet} \ar@{-}@[Brown][rrrrrdd]
&& {\color{Brown}\bullet} \ar@{-}@[Brown][rrrdd]
& \ar@{.}@[Brown][r]
& {\scriptstyle\color{Brown}P^1} \ar@{~}@[Brown][rdd] \ar@{.}@[Brown][rrrrrrr]
& {\color{Brown}\bullet} \ar@{-}@[Brown][dd]
&& {\color{Brown}\bullet} \ar@{-}@[Brown][lldd]
&& {\color{Brown}\bullet} \ar@{-}@[Brown][lllldd]
&&&& {\color{OliveGreen}\ast}
\\ \\
&&&&&&&&& {\color{Brown}\bullet}
&&&&
}}
\end{equation}
This space is barely connected, thanks to the one closed point~$\cM(N)$.
To obtain the spectrum $\Spc(\stperm(Q_8;k))$ of the stable permutation category, we remove all the closed points. The resulting open is disconnected: One piece is~$\Spc(\stperm(V_4;k))$ (as pictured in~\Cref{fig:P1-doubled}) and the other is a single point~$*=\Spc(\stmod(kQ_8))$. As we saw in \Cref{Lem:disconnect}, the disconnection of the spectrum reflects the decomposition of the category $\StPerm(Q_8;k)\cong\StPerm(V_4;k)\times\StMod(kQ_8)$.
\end{Exa}

\begin{Thm}
\label{Thm:cyclic}%
Let $G$ be a finite group whose $p$-Sylow subgroup is cyclic of order~$p^n$ for~$n\ge1$. Choose a tower $1=H_n<H_{n-1}<\ldots<H_1<H_0\le G$ of cyclic subgroups~$H_i$ of order~$p^{n-i}$.
Then the stable permutation category $\StPerm(G;k)$ is equivalent to the product of the stable \emph{module} categories~$\StMod(k(\WGHi))$ for all $i=1,\ldots,n-1$. The components $\StPerm(G;k)\to \StMod(k(\WGHi))$ of this equivalence are given by the (stable) modular $H_i$-fixed points functors $\Psi^{H_i}$ composed with the canonical localization $\bar\Upsilon\colon\StPerm(\WGHi;k)\onto \StMod(k(\WGHi))$.
\end{Thm}

\begin{proof}
The argument is similar to the proof of \Cref{Thm:quaternion}: One proves it for the $p$-Sylow and one reduces to that case by descent. For the $p$-group case (the Sylow) one proceeds by an easy induction on~$n$ using \Cref{Lem:disconnect}, since in this case $G/N$ is cyclic of order~$p^{n-1}$. For the descent argument one uses once again the little miracle about the uniqueness of the subgroup~$H$ of a certain order in the $p$-Sylow, which gives us $N_G(H,P)=N_G(H)$ for every~$H=H_1,\ldots,H_{n-1}$.
Details are left to the interested reader.
\end{proof}

\Cref{Thm:quaternion,Thm:cyclic} specialize to \Cref{Thm:quaternion-intro,Thm:cyclic-intro} in the introduction in the case of $p$-groups.

\begin{Rem}
In the above two proofs, we use central cyclic subgroups~$N\normal P$ of the $p$-Sylow of~$G$, typically their center $N=Z(P)$. Note that we cannot apply \Cref{Lem:disconnect} directly to $N$ in~$G$ as there is no guarantee that $N$ remains normal in~$G$. For instance, for $P=Q_8=\{\pm1,\pm i,\pm j,\pm k\}$ the usual quaternion group, we can find a finite abelian group~$A$ of odd order on which $Q_8$ acts in such a way that the action of~$N=\{\pm1\}$ is non-trivial, for instance the $\bbF_3$-vectorspace~$A=\bbF_{\!\!3}^2$ with the action of~$Q_8$ via the usual embedding $Q_8\hookrightarrow \mathrm{GL}_2(\bbF_{\!\!3})$ given by~$i\mapsto \left(\begin{smallmatrix}0&-1\\1&0\end{smallmatrix}\right)$ and~$j\mapsto \left(\begin{smallmatrix}1&1\\1&-1\end{smallmatrix}\right)$. In such a case, if we let~$G=Q_8\ltimes A$ then the $2$-Sylow $P=Q_8\times\{1\}$ is not normal and the element $(-1,a)$ conjugates~$N$ to a different cyclic subgroup, for any~$a\in A$ not fixed by~$-1$.
\end{Rem}



\begin{thebibliography}{BCR97}

\bibitem[Bal07]{Balmer07}
Paul Balmer.
\newblock Supports and filtrations in algebraic geometry and modular
  representation theory.
\newblock {\em Amer. J. Math.}, 129(5):1227--1250, 2007.

\bibitem[Bal10]{balmer:sss}
Paul Balmer.
\newblock Spectra, spectra, spectra - tensor triangular spectra versus
  {Z}ariski spectra of endomorphism rings.
\newblock {\em Algebraic and Geometric Topology}, 10(3):1521--63, 2010.

\bibitem[Bal12]{Balmer12}
Paul Balmer.
\newblock Descent in triangulated categories.
\newblock {\em Math. Ann.}, 353(1):109--125, 2012.

\bibitem[Bal16]{Balmer16b}
Paul Balmer.
\newblock Separable extensions in tensor-triangular geometry and generalized
  {Q}uillen stratification.
\newblock {\em Ann. Sci. \'Ec. Norm. Sup\'er. (4)}, 49(4):907--925, 2016.

\bibitem[Bal18]{Balmer18a}
Paul Balmer.
\newblock On the surjectivity of the map of spectra associated to a
  tensor-triangulated functor.
\newblock {\em Bull. Lond. Math. Soc.}, 50(3):487--495, 2018.

\bibitem[BCR97]{BensonCarlsonRickard97}
David~J. Benson, Jon~F. Carlson, and Jeremy Rickard.
\newblock Thick subcategories of the stable module category.
\newblock {\em Fund. Math.}, 153(1):59--80, 1997.

\bibitem[BD20]{BalmerDellAmbrogio20}
Paul Balmer and Ivo Dell'Ambrogio.
\newblock {\em Mackey 2-functors and {M}ackey 2-motives}.
\newblock EMS Monographs in Mathematics. European Mathematical Society (EMS),
  Z\"{u}rich, 2020.

\bibitem[BD24]{BalmerDellAmbrogio24}
Paul Balmer and Ivo Dell'Ambrogio.
\newblock Cohomological {M}ackey 2-functors.
\newblock {\em J. Inst. Math. Jussieu}, 23(1):279--309, 2024.

\bibitem[BF11]{balmer-favi:idempotents}
Paul Balmer and Giordano Favi.
\newblock Generalized tensor idempotents and the telescope conjecture.
\newblock {\em Proc. Lond. Math. Soc. (3)}, 102(6):1161--1185, 2011.

\bibitem[BG23a]{BalmerGallauer23a}
Paul Balmer and Martin Gallauer.
\newblock Finite permutation resolutions.
\newblock {\em Duke Math. J.}, 172(2):201--229, 2023.

\bibitem[BG23b]{BalmerGallauer23b}
Paul Balmer and Martin Gallauer.
\newblock Permutation modules, {M}ackey functors, and {A}rtin motives.
\newblock In {\em Representations of algebras and related structures}, EMS Ser.
  Congr. Rep., pages 37--75. EMS Press, Berlin, 2023.

\bibitem[BG25a]{tt-perm}
Paul Balmer and Martin Gallauer.
\newblock The geometry of permutation modules.
\newblock {\em Invent. Math.}, 241(3):841--928, 2025.

\bibitem[BG25b]{BalmerGallauer25a}
Paul Balmer and Martin Gallauer.
\newblock The spectrum of {A}rtin motives.
\newblock {\em Trans. Amer. Math. Soc.}, 378(3):1733--1754, 2025.

\bibitem[Bro82]{brown:cohomology-groups}
Kenneth~S. Brown.
\newblock {\em Cohomology of groups}, volume~87 of {\em Graduate Texts in
  Mathematics}.
\newblock Springer-Verlag, New York-Berlin, 1982.

\bibitem[Del22]{DellAmbrogio22}
Ivo Dell'Ambrogio.
\newblock Green 2-functors.
\newblock {\em Trans. Amer. Math. Soc.}, 375(11):7783--7829, 2022.

\bibitem[Fuh25]{fuhrmann2025modularfixedpointsequivariant}
Yorick Fuhrmann.
\newblock Modular fixed points in equivariant homotopy theory, 2025.

\bibitem[Gal25]{gallauer:periods}
Martin Gallauer.
\newblock Periods in equivariant and motivic contexts, 2025.
\newblock Preprint available at \url{https://arxiv.org/abs/2511.14325}.

\bibitem[Hap88]{Happel88}
Dieter Happel.
\newblock {\em Triangulated categories in the representation theory of
  finite-dimensional algebras}, volume 119 of {\em LMS Lecture Note}.
\newblock Cambr.\ Univ.\ Press, Cambridge, 1988.

\bibitem[Kra05]{Krause05b}
Henning Krause.
\newblock The stable derived category of a {N}oetherian scheme.
\newblock {\em Compos. Math.}, 141(5):1128--1162, 2005.

\bibitem[Nee92]{Neeman92b}
Amnon Neeman.
\newblock The connection between the {$K$}-theory localization theorem of
  {T}homason, {T}robaugh and {Y}ao and the smashing subcategories of
  {B}ousfield and {R}avenel.
\newblock {\em Ann. Sci. \'Ecole Norm. Sup. (4)}, 25(5):547--566, 1992.

\bibitem[Ric89]{Rickard89}
Jeremy Rickard.
\newblock Derived categories and stable equivalence.
\newblock {\em J. Pure Appl. Algebra}, 61(3):303--317, 1989.

\end{thebibliography}


\end{document}